\documentclass[12pt]{amsart}
\usepackage{amssymb,url}

\newtheorem{thm}{Theorem}[section]
\newtheorem{lem}[thm]{Lemma}
\newtheorem{clem}[thm]{Computer Lemma}
\newtheorem{cor}[thm]{Corollary}
\newtheorem{ccor}[thm]{Computer Corollary}
\newtheorem{conj}[thm]{Conjecture}

\def\<{\langle}
\def\>{\rangle}
\def\onto{\twoheadrightarrow}
\def\Z{\mathbb{Z}}

\def\C{\mathbb{C}}
\def\P{PSL(2,\C)}
\def\G{\Gamma}
\def\D{\Delta}
\def\wh{\widehat}
\def\ol{\overline}
\def\V{\mathcal{V}}
\def\X{\mathcal{X}}
\def\a{\alpha}

\newenvironment{spf}{\par\medskip\noindent{\em Sketch Proof.}}{\hfill$\square$\par\medskip}

\begin{document}

\title{Generalised Triangle Groups of Type  $(2,4,2)$}

	\author[Howie]{James Howie }
	\address{ James Howie\\
		Department of Mathematics and Maxwell Institute for Mathematical Sciences\\
		Heriot--Watt University\\
		Edinburgh EH14 4AS }
\email{J.Howie@hw.ac.uk, j53howie@gmail.com}

\maketitle

\begin{abstract}A conjecture of Rosenberger says that a group of the form $\<x,y|x^p=y^q=W(x,y)^r=1\>$ (with $r>1$) is either virtually solvable or contains a non-abelian free subgroup.  This note is an account of an attack on the conjecture in the  case $(p,q,r)=(2,4,2)$.  The results obtained are only partial, but nevertheless provide strong evidence in support of the conjecture in the case in question, in that the word $W$ in any counterexample is shown to satisfy some strong restrictions.  The exponent-sums of $x$ and $y$ in $W$ must be even and  odd respectively, while its free-product (or syllable) length must be at least $68$.  There is also a report of computer investigations which yield a stronger lower bound of $196$ for the free-product length.
\end{abstract}

\section{Introduction}
By a {\em triangle group} we mean the group of isometries of the elliptic, euclidean or hyperbolic plane generated by rotations around the vertices of a triangle whose angles are submultiples $\pi/p,\pi/q,\pi/r$ of $\pi$ (that is, $p,q,r$ are all integers greater than $1$).
It is well known that such a group is linear, and hence satisfies the Tits alternative: either the group is virtually solvable, or it contains a non-abelian free subgroup.  It is also well-known that any triangle group has a presentation of the form
$$\<x,y|x^p=y^q=(xy)^r=1\>.$$
Rosenberger \cite{Ros} defined a {\em generalised triangle group} to be one with a presentation of the form
$$\<x,y|x^p=y^q=W(x,y)^r=1\>,$$
where again $p,q,r$ are integers greater than $1$, and the word $W$ represents a cyclically reduced element of free product length $2k\ge 2$ in the group 
$$\<x,y|x^p=y^q=1\>\cong\Z_p\ast\Z_q.$$
Rosenberger \cite{Ros} conjectured that these groups also satisfy a Tits alternative:

\begin{conj}[\cite{Ros}]\label{Rconj}
If $G$ is a generalised triangle group, then either $G$ is virtually solvable or $G$ contains a non-abelian free subgroup.
\end{conj}

The integer parameters $p,q,r$ in the above definition clearly play an important r\^{o}le in the study of genealized triangle groups, as does the {\em length parameter} $k$.  We shall say that $G$ is {\em of type} $(p,q,r)$, and {\em of length} $k$.  Clearly we may assume without loss of generality that $p\le q$, and that $W$ is not a proper power in $\Z_p\ast\Z_q$.

While  Conjecture \ref{Rconj} in full generality remains an open problem, there is a large body of evidence in support of it, in the form of proofs for suitable values of the parameters.  In particular, the conjecture has been proved if one of the following hold:

\begin{itemize}
\item $p=q=2$ (in which case, of course, $G$ is a finite dihedral group);
\item $k\le 4$ \cite{Ros,LR};
\item $r>2$ \cite{FLR};
\item $q>5$ \cite{LR,Will2,BK,BK2,BKB,BBK2,HW};
\item $p>2$ \cite{BBK,HW342,Ho13}.
\end{itemize}

See the survey \cite{FRR} for more  details.

The three remaining types, $(2,3,2)$, $(2,4,2)$ and $(2,5,2)$, can be attacked using similar methods to those used in the proofs of the results mentioned above, but each has additional inherent difficulties.  Nevertheless some partial results are known for these types \cite{BK03,BK07,Ho10,HK,HK-code}.

In the present paper we concentrate on generalised triangle groups of type $(2,4,2)$. and provide additional evidence in support of Rosenberger's Conjecture \ref{Rconj} by proving the following.

\begin{thm}\label{longce}
Suppose that
$$W=W(x,y)=xy^{a(1)}xy^{a(2)}\cdots xy^{a(k)}$$ with $a(j)\in\{1,2,3\}$ for each $j$, and that
$$G:=\<x,y\mid x^2=y^4=W(x,y)^2=1\>$$
is a counterexample to Conjecture \ref{Rconj}.
Then
\begin{itemize}
\item $a(j)=2$ for precisely one value of $j$;
\item $k$ has the form $2^ms$ with $m\ge 1$ and $s$ odd; and
$s\ge 17$ (hence $k\ge 34$ and $W$ has free-product length $\ge 68$).
\end{itemize}
\end{thm}

The study of generalised triangle groups of type $(2,4,2)$ splits naturally into four subcases, depending on the parity of the exponent-sums of $x,y$ in $W$.  The subcase where $x,y$ both have odd exponent sum is due to Benyash-Krivets \cite{BK07}; for completeness we include a slightly different proof -- see Theorem \ref{oddoddthm} below.   Benyash-Krivets also made significant progress in \cite{BK03} towards a proof in the subcase where $x,y$ both have even exponent sum; our proof in this subcase makes extensive use of the methods developed there.

Theorem \ref{longce} successfully deals with three of the four subcases.  Our methods have not proved sufficient to deal with the fourth subcase, in which the exponent-sums are odd and even respectively.  Nevertheless, the length inequality in Theorem \ref{longce} is very suggestive that the conjecture might also hold in that subcase.

Our proof of Theorem \ref{longce} is purely theoretical.  The methods behind the proof also serve to eliminate certain other values of the odd number $s$ associated to a putative counterexample -- including $21,23,29,31,45$ and $47$.   If we allow the use of computer searches as part of a proof technique, then the lower bound for $s$ can be significantly improved.  Indeed, by using small cancellation theory (a good technique for finding free subgroups) to streamline our searches, but running these only on a standard laptop computer in relatively short sessions (a few hours), we have obtained a lower bound of $49$ for $s$ -- and hence lower bounds of $98$ for $k$ and $196$ for the free-product length of $W$.   More detail on these computer searches are given in \S \ref{it} below.   Results in \S \ref{it} in which the proof involves computer methods have been labelled as {\bf Computer Lemma} or {\bf Computer Corollary} to distinguish them from results that have been explicitly proved by theoretical arguments.

\medskip
As usual with generalised triangle groups, our principle method is the consideration of {\em essential representations} into $\P$, that is to say, homomorphisms $\rho:G\to\P$ with the property that $\rho(x),\rho(y),\rho(W)$ have orders $p,q,r$ respectively.  These are (with a few well-understood exceptions) determined up to conjugacy by a polynomial $\tau_W(\lambda)\in\C[\lambda]$, called the {\em trace polynomial} or {\em Fricke polynomial} whose roots form the  (zero-dimensional) character variety corresponding to this class of representations.   

We give an analysis of character values of essential representations, and of trace polynomials, in \S \ref{charvar} below.  In particular we apply a trick due to Benyash-Krivets \cite{BK03}, together with some other observations, to give strong restrictions on the form of trace polynomial that can arise in any putative counterexample to the Rosenberger Conjecture of type $(2,4,2)$.

Two other techniques give rise to further partial results.  In \S \ref{intandpar} we show that the trace polynomial $\tau_W(\lambda)$ belongs either to $\Z[\lambda]$ or $\sqrt2\Z[\lambda]$, depending on the exponent-sums of $x$ and $y$ in $W$. In the former case we give restrictions on the parity of the coefficients, which in turn lead to restrictions on $W$ in any counterexample.   Then in \S \ref{largeness} we exploit a criterion for a group to be large -- that is, to contain a finite index subgroup admitting a non-abelian free homomorphic image -- to obtain further restrictions on $W$.  

The results of \S\S \ref{charvar} - \ref{largeness} suffice to prove the first two parts of Theorem \ref{longce}, and reduce us to consideration of words $W$ with trace polynomial $\tau_W(\lambda)=\sqrt2(\lambda^2-1)^{k/2}$.  The lower bound of $17$ for the odd part $s$ of $k$ is proved in \S \ref{evenodd}, after first eliminating certain short words $W$ by ad-hoc methods in \S \ref{short}.  

Finally, in \S \ref{it}, we discuss the logistics of computer searches that give rise to the increased lower bound of $49$ for $s$.

\section{Character variety methods}\label{charvar}

We consider the $\P$-character variety $\X_G$ of representations $\rho$ from a finitely generated group $G$ to $\P$.  If $G$ is $n$-generated, then the epimorphism $F_n\to G$ induces an embedding of $\X_G$ into $\X_{F_n}$, so it is useful to understand the character varieties of free groups.  In practice we need only the cases $n=2$ and $n=3$: the latter in order to apply the method developed by Benyash-Krivets \cite{BK03}, using the character variety of an index $2$ subgroup of the generalised triangle group when $y$ has even exponent-sum in $W$.  

Now any representation $\rho:F_n\to\P$ lifts to a representation to $SL(2,\C)$, and the $SL(2,\C)$-characters of $F_2$ and $F_3$ are well-understood through 19'th century work of Fricke and others.  An accessible exposition of these results is available in \cite{Gold}, for example.  We recall here only the basic facts which we will require in the sequel.

 If $\{x_1,x_2\}$ is a basis for $F_2$, then  the $SL(2,\C)$-character variety of $F_2$ is a $3$-dimensional affine space with basis consisting of the traces of $\rho(x_1),\rho(x_2)$ and $\rho(x_1x_2)$.  Thus the character of an essential $\P$- representation $\rho$ of a generalised triangle group $$\<x,y|x^2=y^4=W(x,y)^2=1\>$$ corresponds to a point $(\alpha,\beta,\lambda)$ of $\X_{F_2}$ such that
$\alpha=Tr(\rho(x))=0$ since $\rho(x)$ has order $2$, $\beta=Tr(\rho(y))=\pm\sqrt2$ since $\rho(y)$ has order $4$, and $Tr(\rho(W))=0$ since $\rho(W)$ has order $2$.   In practice we will fix $\beta=+\sqrt2$: the trace of $\rho(W)$ is then given by a polynomial $\tau_W(\lambda)$ in the variable $\lambda:=Tr(\rho(xy))$, which we call the {\em trace polynomial} of $W$. 

If $\{x_1,x_2,x_3\}$ is a basis for $F_3$, then   the $SL(2,\C)$-character variety of $F_3$ is a $6$-dimensional hypersurface in the $7$-dimensional affine space with basis consisting of the traces of $\rho(x_1)$, $\rho(x_2)$, $\rho(x_3)$, $\rho(x_1x_2)$, $\rho(x_1x_3)$, $\rho(x_2x_3)$ and $\rho(x_1x_2x_3)$.  It is defined by a single polynomial in these variables, which is quadratic in $Trace(\rho(x_1x_2x_3))$.

\medskip
Rosenberger \cite{Ros} classifies finitely generated subgroups $\G<\P$ without non-abelian free subgroups into three mutually exclusive types:
\begin{itemize}
\item[(a)] $\G$ is reducible; 
\item[(b)] $\G$ is irreducible, but contains a reducible abelian subgroup $\Delta$ of index $2$, and every element of $\G\setminus\Delta$ has order $2$;
\item[(c)] $\G$ is conjugate to $A_4$, $S_4$ or $A_5$.
\end{itemize} 

If $G=gp\{x,y\}$ is a $2$-generator group without non-abelian free subgroups, then Rosenberger's classification splits the character variety $\X_G$ of representations  $\rho:G\to\P$ into a disjoint union of three sets $\V(a)$, $\V(b)$ and $\V(c)$, depending on whether $\rho(G)$ is of type (a), (b), or (c).  
The character $[\rho]$ of a representation $\rho:G\to\P$ belongs to $\V(a)$ if and only if $\rho(xyx^{-1}y^{-1})$ has trace $2$. Hence $\V(a)$ is defined by a single trace equation and so is Zariski closed.

The character $[\rho]$ belongs to $\V(b)$ if and only if $[\rho]\notin\V(a)$ and at least two of $\rho(x)$, $\rho(y)$, $\rho(xy)$ have order $2$ in $\P$ -- and hence trace $0$.   Thus $\V(a)\cup\V(b)$ is the union of four sets, each defined by at most two trace equations, so it is also Zariski closed.

It follows that $\V(c)=\X_G\setminus(\V(a)\cup\V(b))$ is Zariski-open in $\X_G$.

If $\V$ is an irreducible component of the character variety $\X_G$, and contains the characters of representations $\rho_1,\rho_2$ of type (c),
then $\V\cap\V(c)$ is a non-empty, Zariski-open subset of $\V$, and hence dense in $\V$.  Since traces of matrices representing elements of $A_4\cup S_4\cup A_5$ belong to a finite set, and the trace function is continuous on $\X_G$, it follows that $Tr\rho(g)$ is constant on $\V$ for each $g\in G$.  In particular, $\rho_1,\rho_2$ have the same character: and $\V$ is a single point.  Hence:

\begin{lem}\label{dim0}
If $G$ is a  $2$-generator group without non-abelian free subgroups, and admits one of $A_4,S_4,A_5$ as a homomorphic image, then the $\P$-character variety of $G$ is $0$-dimensional.
\end{lem}

The following consequence of Lemma \ref{dim0} is essentially due to Benyash-Krivets \cite{BK03}.

\begin{cor}\label{Honto}
Let $$G:=\<x,y\mid x^2=y^4=W(x,y)^2=1\>$$
where $W:=xy^{a(1)}xy^{a(2)}\cdots xy^{a(k)}$ with $\sum_j a(j)$ even.  Let $K<G$ be the subgroup generated by  $u:=x$, $v:=yxy^{-1}$ and $z:=y^2$; and let $H<K$ the subgroup generated by $\{uz,zv\}$.  If $G$ or $K$ or $H$  has one of $A_4,S_4,A_5$ as a homomorphic image, then $G$ contains a non-abelian free subgroup.
\end{cor}

\begin{proof}
Now $|G:K|=|K:H|=2$.  The image of $K$ under an epimorphism from $G$ onto one of $A_4,S_4,A_5$ has image of index at most $2$, so is also isomorphic to one of $A_4,S_4,A_5$.  Similarly any epimorphism from $K$ onto one of $A_4,S_4,A_5$ restricts to an epimorphism from $H$ onto one of $A_4,S_4,A_5$.  So we may assume that such a representation of $H$ exists.  

Now $K$ has a presentation of the form
$$\<u,v,z|u^2=v^2=z^2=W_1^2=W_2^2=1\>,$$
where $W_1,W_2$ are rewrites of $W,yWy^{-1}$ in terms of $\{u,v,z\}$.   

An {\em essential} representation $\rho:K\to\P$ satisfies 
\begin{equation}\label{charK}
Tr\rho(u)=Tr\rho(v)=Tr\rho(z)=Tr\rho(W_1)=Tr\rho(W_2)=0.
\end{equation}
Conversely, any homomorphism $\<u,v,z\> \to\P$ satisfying \eqref{charK} defines an essential representation of $K$ to $\P$. So the subvariety  $\mathcal{EX}_K\subseteq\mathcal{X}_K$ of essential characters has dimension $\ge 1$.

Note that, of the seven trace parameters which determine the character $[\rho]$ of an essential $\P$-representation of $K$, three (the traces of $u,v,z$) are zero.   Of the four remaining trace parameters, three (the traces of $uv,uz,vz$) determine the character of the restriction of $\rho$ to $H$, and the last (the trace of $uvz$) satisfies a quadratic equation in terms of the first six.  Hence restricting such representations to $H$ gives a polynomial map $\Phi$ from $\mathcal{EX}_K$ to the $\P$-character variety $\mathcal{X}_H$ of $H$ which is at most $2$-to-$1$.  Since  $\mathcal{EX}_K$ has dimension $\ge 1$, it follows that  $\mathcal{X}_H$  also has dimension $\ge 1$. 

By the Lemma, $H$ must have a non-abelian free subgroup, and hence so also has $G$.
\end{proof}

\begin{cor}\label{notac}
Let $G$ be a generalised triangle group of type $(2,4,2)$ admitting $\P$-representations of types (a) and (c).  Then $G$ has a non-abelian free subgroup.
\end{cor}

\begin{proof}
Immediate from Corollary \ref{Honto}, since $$G=\<x,y|x^2=y^4=W^2=1\>$$ has an essential $\P$-representation of type (a) only if $y$ has even exponent-sum in $W$.
\end{proof}

\begin{thm}\label{eveneventhm}
Let $$G:=\<x,y\mid x^2=y^4=W(x,y)^2=1\>$$
where $W:=xy^{a(1)}xy^{a(2)}\cdots xy^{a(k)}$ with $a(j)\in\{1,3\}$ for each $j$, and with $k\ge 6$ even.  Then $G$ has a non-abelian free subgroup.
\end{thm}

\begin{proof}
Suppose not.  Then by Corollary \ref{Honto} there is no representation $H\to\P$ of type (c).  In the notation of the proof of Corollary \ref{Honto}, we consider the image of the Zariski-open subset 
$$\mathcal{U}_1:=\{[\rho]\in\mathcal{EX}_K, Tr\rho(uvz)\ne 0\ne Tr\rho(uv)\}$$
 under the restriction map $\Phi:\mathcal{EX}_K\to\mathcal{X}_H$ (where $[\rho]$ denotes the character of the representation $\rho$).  The restriction to $H$ of an essential representation $\rho:K\to\P$ is reducible if and only if 
$$2=Tr\rho((uz)(zv)(uz)^{-1}(zv)^{-1})=-Tr\rho((uvz)^2)=2-Tr\rho(uvz)$$  
(using the fact that $\rho(u^2)=\rho(v^2)=\rho(z^2)=-I$).
So for $[\rho]\in\mathcal{U}_1$ we must have $\rho|_H$ of type (b).  Moreover, since 
$$Tr\rho((uz)(zv))=-Tr\rho(uv)\ne 0,$$
 it follows that $\mathcal{U}_1$ is contained in the Zariski-closed subset 
$$\mathcal{U}_2:=\{[\rho]\in\mathcal{EX}_K,Tr\rho(uz)=Tr\rho(zv)=0\}.$$
We complete the proof by showing that $\mathcal{U}_2$ is finite, while $\mathcal{U}_1$ is non-empty, and hence $1$-dimensional.  This gives a contradiction.

To see that $\mathcal{U}_2$ is finite, note that the conditions 
$$Tr\rho(u)=Tr\rho(v)=Tr\rho(z)=Tr\rho(uz)=Tr\rho(zv)=0$$
 imply that the image of $[\rho]\in\mathcal{U}_2$ in $\mathcal{X}_H$ is determined by the single parameter $Tr\rho(uv)$.  They also imply that $\rho(z)$ commutes with each of $\rho(u)$ and $\rho(v)$. So the defining equations 
$$Tr\rho(W_1)=Tr\rho(W_2)=0$$ in $\mathcal{EX}_K$ are equivalent to $$Tr(\rho((uv)^{k/2}z^s))=Tr(\rho(
(vu)^{k/2}z^{k-s}))=0$$ for some integer $s$.   Hence $\rho((uv)^k)=\pm I$ (depending on the parity of $s$), and so only finitely many values for $Tr\rho(uv)$ are possible.

By the same analysis, to show that $\mathcal{U}_1$ is non-empty, it suffices to solve the matrix equations $$Tr(U)=Tr(V)=Tr(Z)=Tr(UZ)=Tr(VZ)$$ $$=Tr((UV)^{k/2}Z^s)=Tr((VU)^{k/2}Z^{k-s})=0,$$ together with the inequalities $$Tr(UV)\ne 0\ne Tr(UVZ),$$ in $SL(2,\C)$.  Since $k\ge 6$ we can do this:
put
$$U:=\left(\begin{array}{cc} 0 & 1\\ -1 & 0\end{array}\right),\qquad
V:=\left(\begin{array}{cc} 0 & \omega\\ -\ol\omega & 0\end{array}\right),\qquad
Z:=\left(\begin{array}{cc} i & 0\\ 0 & -i\end{array}\right),$$
where $\omega:=e^{2i\pi/k}$ if $s$ is odd, or $\omega:=e^{i\pi/k}$ if $s$ is even.  
\end{proof}

\noindent{\bf Remark}
The inequality $k\ge 6$ in Theorem \ref{eveneventhm} cannot be removed: $\<x,y|x^2=y^4=(xyxyxyxy^3)^2=1\>$ is virtually abelian \cite[Lemma 7]{LR}.

\section{Integrality and parity methods}\label{intandpar}
If $$G=\<x,y|x^2=y^4=W(x,y)^2=1\>$$ is a generalised triangle group of type $(2,4,2)$, then its essential $\P$-character variety consists of the set of roots of the trace polynomial $\tau_W(\lambda):=Tr(W(X,Y))$, where $X,Y\in SL(2,\C)$ with $Tr(X)=0$, $Tr(Y)=\sqrt2$ and $Tr(XY)=\lambda$.

Some useful properties of the trace polynomial are summarised in the following:

\begin{lem}\label{parity}
\begin{enumerate}
\item $\tau_W(\lambda)$ is an odd or even polynomial in $\lambda$, depending on the parity of the length parameter $k$.  
\item If $k>0$ then its leading coefficient is $\sqrt2^a$ where  $a$ denotes the number of $y^2$-letters in $W$.  
\item $\tau_W(\lambda)\in\Z[\lambda]$ if $a$ is even; and $\tau_W(\lambda)\in\sqrt2\Z[\lambda]$ if $a$ is odd.
\item If $\lambda$ is any real number with $-\sqrt2\le\lambda\le\sqrt2$, then $-2\le\tau_W(\lambda)\le 2$.
\end{enumerate}
\end{lem}

\begin{proof}
For the first three statements we argue by induction on $k$, making extensive use of the standard trace identity
\begin{equation}\label{trid}
Trace(UV)+Trace(UV^{-1})=Trace(U)\cdot Trace(V).
\end{equation}
Thus in particular $\tau_{xy^2}(\lambda)=Trace(XY)\cdot Trace(Y)-Trace (X)=\sqrt2\lambda$ and $\tau_{xy^3}(\lambda)=Trace(XY^2)Trace(Y)-Trace(XY)=\lambda$. So the result holds for $k=1$.  It trivially holds for $k=0$ since $\tau_1(\lambda)=Trace(I)=2$.

For the inductive step, we write $W=UV$ where $U=xy^{\a(1)}\cdots xy^{\a(k-1)}$ and $V=xy^{\a(k)}$.  Then $UV^{-1}$ is conjugate in $\Z_2*\Z_4$ to a word of length parameter less than $k$ but congruent to $k$ modulo $2$.  So the result follows from \eqref{trid} by applying the inductive hypothesis to $U$ and to $UV^{-1}$.

For the final statement we choose matrices $X,Y\in SU(2)$ of traces $0,\sqrt2$ respectively, such that $XY$ has trace $\lambda$.  Then $W(X,Y)\in SU(2)$ so has real trace in $[-2,2]$.
\end{proof}

\begin{cor}\label{tauWform}
Let $exp_y(W)$ denote the exponent-sum of $y$ in $W$.
If $G$ has no non-abelian free subgroups, then $\tau_W(\lambda)$ has the form
\begin{eqnarray*}
\sqrt2^a\lambda^b(\lambda^2-1)^c\qquad\mathrm{if}~exp_y(W)~\mathrm{is~odd;}\\
\sqrt2^a\lambda^b(\lambda^2-2)^c\qquad\mathrm{if}~exp_y(W)~\mathrm{is~even}
\end{eqnarray*}
for some natural numbers $a,b,c$ with $a$ as in Lemma \ref{parity}.
\end{cor}

\begin{proof}
Since $G$ has no non-abelian free subgroups, Rosenberger's classification shows that any essential $\P$-representation has image of type (a), (b) or (c).  A representation of type (a) corresponds to roots $\pm\sqrt2$ of $\tau_W$, one of type (b) to roots $0$, and one of type (c) to roots $\pm1$ -- since its image has an element of order $4$ and so cannot be isomorphic to $A_4$ or $A_5$.  By Lemma \ref{parity} $\tau_W$ is either odd or even so $+\sqrt2$ and $-\sqrt2$ occur with equal multiplicity, as do $+1$ and $-1$.
Moreover, $\pm\sqrt2$ and $\pm1$ cannot both occur as roots of $\tau_W$ by Lemma \ref{notac}. To complete the proof, note that the leading coefficient of $\tau_W(\lambda)$ is given by Lemma \ref{parity}.
\end{proof}

\begin{cor}\label{OKmod2}
If the number $a$ in Lemma \ref{parity} is even, and $W'$ is obtained from $W$ by changing a $y$-letter to a $y^3$-letter, or vice versa, then $\tau_{W'}(\lambda)\equiv\tau_W(\lambda)$ modulo $2$.
\end{cor}

\begin{proof}
Up to conjugacy, and interchanging $W,W'$, we may assume that $W,W'$ have the form $Uxy$ and $Uxy^3$ respectively, for some common subword $U$.  Writing $V:=Uxy^2$ we have $W=Vy^{-1}$ and $W'=Vy$, and $\tau_V(\lambda)\in\sqrt2\Z[\lambda]$ by Lemma \ref{parity}, so the standard trace identity \eqref{trid} gives
$$\tau_W(\lambda)+\tau_{W'}(\lambda)=\tau_y(\lambda)\cdot\tau_V(\lambda)=\sqrt2\cdot\tau_V(\lambda)\in 2\Z[\lambda].$$
\end{proof}

\begin{cor}\label{2ysquareds}
If $a=2$, and $\a(m)=\a(n)=2$ where $1\le m<n\le k$, then the following are equivalent:
\begin{enumerate}
\item $\tau_W(\lambda)\in 2\Z[\lambda]$;
\item  $k=2(n-m)$;
\item $W$ belongs to the normal closure of $y^2$ in $\<x,y~|~x^2=y^4=1\>$.
\end{enumerate}
\end{cor}

\begin{proof}
Up to conjugacy, we may assume that $n=k$ and $m\ge k/2$.  By Corollary \ref{OKmod2}, replacing a letter $y$ by $y^3$ or vice versa does not affect the trace polynomial modulo $2$.  Nor does it change the coset $WN$ where $N$ is the normal closure of $y^2$.

So we may assume that $\a(j)=1$ for $j\notin\{m,n\}$, 
so $W=UV$ with $U=(xy)^{m-1}xy^2$ and $V=(xy)^{k-m-1}xy^2$.  The standard trace identity \eqref{trid} then gives 
$$\tau_W(\lambda)=\tau_U(\lambda)\cdot\tau_V(\lambda)-\tau_{UV^{-1}}(\lambda).$$
But $\tau_U(\lambda),\tau_V(\lambda)\in\sqrt2\Z[\lambda]$, while $UV^{-1}=(xy)^{2m-k}$.  It follows that $\tau_W(\lambda)\in 2\Z[\lambda]$ if and only if $U=V$ if and only if $2m=k$, as required.

An easy exercise shows that $WN=(xy)^{2m-k}N$, so $W\in N$ if and only if $2m=k$, as required.
\end{proof}

Benyash-Krivets \cite{BK07} proved a special sub-case of the Rosenberger Conjecture using a parity argument.  For completeness we include here a slightly different proof, also using a parity argument, but based on Corollary \ref{OKmod2}.

\begin{thm}\label{oddoddthm}
Let $G=\<x,y|x^2=y^4=W(x,y)^2=1\>$ be a generalised triangle group of type $(2,4,2)$ in which the exponent sums of $x,y$ respectively in $W$ are both odd, and in which the length parameter $k$ is greater than $3$.  Then $G$ contains a non-abelian feee subgroup.
\end{thm}

\begin{proof}
Since $exp_y(W)$ is odd, we may assume that 
$\tau_W(\lambda)$ has the form $\sqrt2^a\lambda^b(\lambda^2-1)^c$ for some natural numbers $a,b,c$, by Corollary \ref{tauWform}.  Here $k=b+2c$ is odd, and $a$ is the number of $y^2$-letters in $W$, which is even since $y$ has odd exponent sum in $W$.  Finally, $\sqrt2^{a+b}=\tau_W(\sqrt2)\le 2$ by Lemma \ref{parity}, so $a+b\le 2$.  It follows that $a=0$ and $b=1$, so $\tau_W(\lambda)=\lambda(\lambda^2-1)^c$ and $k=2c+1$.

\medskip
Since $a=0$, $\tau_W(\lambda)$ is congruent modulo $2$ to $\tau_V(\lambda)$ where $V=(xy)^{2c+1}$, by Corollary \ref{OKmod2}.  In particular, the coefficient $c$ of $\lambda^{k-2}$ in $\tau_W(\lambda)$ is congruent modulo $2$ to that of $\tau_V(\lambda)$ which is $2c+1$.  So $c$ is odd and $k=2c+1\equiv 3$ modulo $4$.  But then $k\ge 7$ since $k>3$.  Continuing in the same way, we can compare the parities of the coefficients of $\lambda^{k-4}$ and $\lambda^{k-6}$ in $\lambda(\lambda^2-1)=\tau_W(\lambda)$ and $\tau_V(\lambda)$.  A calculation gives a triple of simultaneous congruences:
$$c\equiv 2c+1~~\mathrm{modulo}~~2; ~~\frac{c(c-1)}2\equiv\frac{(2c+1)(2c-2)}2~~\mathrm{modulo}~~2;$$
$$\frac{c(c-1)(c-2)}6\equiv\frac{(2c+1)(2c-3)(2c-4)}6~~\mathrm{modulo}~~2.$$
It is an easy exercise to show that this set of simultaneous congruences has no solution.
\end{proof}

\section{Largeness methods}\label{largeness}

Recall that a group $G$ is {\em large} if it has a finite-index subgroup that has a free homomorphic image of rank at least $2$.
In particular, every large group contains a free subgroup of rank $2$ (but not conversely: for example Higman's famous $4$-generator, $4$-relator group with no non-trivial finite images \cite{Hi} is constructed as a free product amalgamated over a free group of rank $2$).   Various criteria for largeness exist.  We shall use the following. 

\begin{lem}\label{large}
Let $$G=\<x_1,\dots,x_m~|~ r_1=r_2=\cdots=r_n=r_{n+1}^p=\cdots=r_{n+q}^p\>$$
where $p$ is a prime number, $r_1,\dots,r_{n+q}$ are words in $\{x_1,\dots,x_m\}$ and $m>n+1$.  If $G/[G,G]$ is infinite then $G$ is large.
\end{lem}

\begin{proof}

Let $\<t\>$ be an infinite cyclic group, and $\Lambda=\Z_p[t,t^{-1}]$ be its group ring over the field $\Z_p$, in other words the ring of Laurent polynomials in one variable with $\Z_p$-coefficients.  Let $\psi:G\twoheadrightarrow\<t\>$ be an epimorphism.

Then the modulo $p$ {\em Alexander matrix} of this presentation of $G$ (with respect to $\psi$) is the $\Lambda$-matrix whose $i,j$ entry is the Fox derivative of the $j$'th relation of the presentation with respect to the  $i$'th generator, evaluated in $\Lambda$ via the natural ring epimorphism $\Z G\twoheadrightarrow\Lambda$ induced from $\psi$.  The modulo $p$ {\em Alexander polynomial} of $G$ (with respect to $\psi$) is the greatest common divisor of the $(m-1)\times(m-1)$ minors of this matrix.  It is defined only up to multiplication by trivial units in $\Lambda$, but is otherwise independent of the presentation of $G$.

In our example, several of the relators are $p$-th powers of words which represent elements of the kernel of $\psi$.  Any Fox derivative of such a relator is a multiple of $p$ and so vanishes in $\Lambda$.   Hence our modulo $p$ Alexander matrix has at most $n<m-1$ non-zero columns, so its $(m-1)\times(m-1)$ minors are all zero.  Hence the modulo $p$ Alexander polynomial is also zero. Largeness of $G$ then follows from a result of Button \cite[Theorem 2.1]{Bu}.
\end{proof}

We will apply this result in two different subcases of the Rosenberger Conjecture.

\begin{thm}\label{twotwosthm}
Let $G=\<x,y|x^2=y^4=W(x,y)^2=1\>$ be a generalised triangle group of type $(2,4,2)$, where two or more $y^2$ letters occur in $W$.  Then $G$ contains a non-abelian free subgroup.
\end{thm}

\begin{proof}
We assume that $G$ has no non-abelian free subgroups, and derive a contradiction as follows.  By Lemma \ref{tauWform} $\tau_W(\lambda)$ has the form $\sqrt2^a\lambda^b(\lambda^2-1)^c$ if $exp_y(W)$ is odd, or $\sqrt2^a\lambda^b(\lambda^2-2)^c$ if $exp_y(W)$ is even, where $a,b,c$ are natural numbers, and $a\ge 2$ by Lemma \ref{parity}.

Suppose first that $exp_y(W)$ is odd.  Then $\sqrt2^{a+b}=\tau_W(\sqrt2)\le2$, by Lemma \ref{parity}. It follows that $a=2$ and $b=0$, whence $k=2c$ is even, and hence also $exp_y(W)$ is even, a contradiction.

Hence $exp_y(W)$ is even and the trace polynomial then has the form $\tau_W(\lambda)=\sqrt2^a\lambda^b(\lambda^2-2)^c$ with $a\ge 2$.  Moreover $\sqrt2^a=|\tau_W(1)|\le 2$ by Lemma \ref{parity}.  It follows that $a=2$ and $\tau_W(\lambda)=2\lambda^b(\lambda^2-2)^c\in 2\Z[\lambda]$.

Now by Corollary \ref{2ysquareds} it follows that $W$ belongs to the normal closure of $y^2$ in $\Z_2*\Z_4$, so $G\twoheadrightarrow \Z_2*\Z_2$.  If $H$ is the normal closure of $y^2$ and $xy$ in $G$, then $H$ has infinite abelianisation, and a presentation of the form
$$\<u,v~|~v^2=W_1^2=W_2^2=1\>,$$
where $v=y^2$, $u=xy$, and $\{W_1,W_2\}=\{W,xWx\}$, rewritten in terms of $u,v$.
The result now follows from Lemma \ref{large}.
\end{proof}

\begin{thm}\label{oddeventhm}
Let $G=\<x,y|x^2=y^4=W(x,y)^2=1\>$ be a generalised triangle group of type $(2,4,2)$, where $W$ has odd length parameter $k\ge 5$ and $y$ has even exponent sum in $W$.  Then $G$ contains a non-abelian free subgroup.
\end{thm}

\begin{proof}
Suppose that $G$ has no non-abelian free sugroups.  Then by Corollary \ref{tauWform}, $\tau_W(\lambda)=\sqrt2^a\lambda^b(\lambda^2-2)^c$ for some integers $a,b,c$ where $k=b+2c$ and $a$ is the number of $y^2$ letters in $W$.  Since $k$ is odd and $y$ has even exponent sum, it follows that $a$ is odd, and by Theorem \ref{twotwosthm} we may assume that $a=1$.  Clearly $b$ is odd since $k$ is odd.

In particular, $b\ge 1$, so there exists an essential representation $\rho$ from $G$ onto the dihedral group of order $8$, corresponding to the root $0$ of $\tau_W(\lambda)$.  Since the exponent sums of $x,y$ in $W$ are odd and even respectively, it follows that $\rho(W)$ is conjugate to $\rho(x)$.  The subgroup $H:=\rho^{-1}(\<\rho(x)\>)$ of $G$ has index $4$ and a Schreier transversal $\{y^j~;~0\le j\le 3\}$.  The corresponding Reidemeister-Schreier presentation for $H$ therefore has the form
$$\<x_0,x_1,x_2,x_3~|~x_0^2=x_2^2=x_1x_3=W_0^2=W_2^2=W_1W_3=1\>,$$
where $x_j=y^jxy^{-j}$ and $W_j=y^jWy^{-j}$.  Since this presentation has $4$ generators and only $2$ relators that are not squares, it follows from Theorem \ref{large} that $H$ is large, provided it has infinite abelianisation.  It remains to prove that the abelianisation of $H$ is indeed infinite.  We do so using the exact sequence
$$0 \to K/[K,H] \to H/[H,H] \to H/K \to 0,$$ 
where $K$ denotes the kernel of $\rho$.

%There exist matrices $X,Y\in SU(2)$ with $Tr(X)=0$, $Tr(Y)=\sqrt2$ and $Tr(XY)=1/\sqrt2$.  Then $W(X,Y)\in SU(2)$ so
By Lemma \ref{parity} we have
$$2\ge% |Tr(W(X,Y)|=
|\tau_W(1/\sqrt2)|=\sqrt2^{a-b}(3/2)^c.$$
It follows that $b\ge 3$, since otherwise $b=1$, so $c\ge 2$ and the above inequality reads $2\ge 9/4$, clearly a contradiction.

Let $\Lambda:=\C[\lambda]/\<\<(\lambda^2)\>\>$.  Then we can lift $\rho$ to a representation $\wh\rho:G\to PSL(2,\Lambda)$ given by matrices
$$X:=\begin{pmatrix} 0 & -1\\ 1 & 0\end{pmatrix},\qquad Y:=\begin{pmatrix} \zeta & \lambda \\ 0  & \ol\zeta\end{pmatrix},$$
where $\zeta$ is a primitive $8$'th root of unity in $\C$.

Then $\wh\rho(K)$ is generated by the conjugates of $$-\wh\rho((xy)^2)=I-\lambda(XY)\equiv I+\lambda\begin{pmatrix} 0 & \ol\zeta\\ -\zeta & 0\end{pmatrix}~~\mathrm{mod}~~\lambda^2$$ by powers of $$\wh\rho(y)\equiv\begin{pmatrix} \zeta & 0 \\ 0  & \ol\zeta\end{pmatrix}~~\mathrm{mod}~~\lambda .$$

Since $(I+\lambda A)(I+\lambda B)\equiv I+\lambda(A+B)$ modulo $\lambda^2$ for any matrices $A,B$, it follows that $\wh\rho(K)$ is free abelian.  Indeed a closer inspection reveals that $\wh\rho(K)\cong\Z^2$, and that conjugation by  $\wh\rho(x)$ acts by interchanging a basis pair.  Hence $K/[K,H]$ has an infinite homomorphic image, which completes the proof.
\end{proof}

\section{Short Words}\label{short}

As mentioned in the Introduction, the Rosenberger Conjecture has been verified for words of length parameter $k\le 4$ \cite{Ros,LR}.   This, combined with Theorems \ref{eveneventhm}, \ref{oddoddthm},  \ref{twotwosthm} and \ref{oddeventhm}, verifies the Rosenberger Conjecture for generalised triangle groups of type $(2,4,2)$ for which either the length parameter $k$ is odd or the exponent-sum of $y$ is even.   We examine the remaining sub-case in \S \ref{evenodd} below, where the following three groups $G_j:=\<x,y|x^2=y^4=W_j^2=1\>$ ($j=1,2,3$) are of particular interest.

\medskip
\begin{center}
\begin{tabular}{|l|l|l|}
\hline
$j$ & $W=W_j$ & $\tau_W(\lambda)$ \\
\hline\hline
$1$ & $xyxyxy^3xy^3xyxy^2$ & $\sqrt2(\lambda^2-1)^3$  \\
\hline
$2$ & $xyxyxy^3xy^3xyxy^3xy^3xy^2$ & $\sqrt2(\lambda^2-1)^4$ \\
\hline
$3$ & $xyxy^3xy^3xyxyxyxy^3xy^2$ & $\sqrt2(\lambda^2-1)^4$ \\
\hline
\end{tabular} 
\end{center}

In his PhD thesis \cite{Wil}, Williams extended the results of \cite{Ros} and \cite{LR} to cover words of length parameter $k\le 6$, with a small number of exceptions -- none of which are of type $(2,4,2)$ -- which his methods did not at the time cover.  (Some of these exceptions are covered by more recent results.)   In particular, he showed that the group $G_1$ above is large.  Indeed, it turns out that all three of these groups are large.  This can easily be checked by computer, but in the interests of completeness (and because \cite{Wil} is unpublished) we give a brief sketch proof here.

\begin{lem}\label{W123}
The three groups $G_1,G_2,G_3$ are large.
\end{lem}

\begin{spf}

\medskip\noindent $\mathbf{G_1}$

We can identify the kernel of the epimorphism $\Z_2*\Z_4\onto S_4$ with the fundamental group of a graph $\G$, the $1$-skeleton of the octahedron.  Paths corresponding to words in $x,y$ can be traced out in $\G$ according to the rule that $x$-letters mean travel along an edge, while $y,y^2,y^3$ letters are instructions to turn right, carry straight on, and turn left respectively, in passing from one edge to the next. The relator $W_1^2$ of $G_1$ lifts to twelve relations in the kernel $K_1$ of $G_1\onto S_4$.  Tracing these according to the rule, each relation identifies a clockwise triangle of the octahedron (a conjugate of $(xy)^3$ in $\Z_2*\Z_4$) with a conjugate of the anti-clockwise antipodal triangle.  

To show that $G_1$ is large, we 
map $\G$ onto a smaller graph $\D$: choose an edge $e\in\G$ and shrink it to a point.  For each of the two triangles containing $e$, fold the other two edges of the triangle together.  Now do the same at the edge antipodal to $e$ in $\G$.
The resulting graph $\D$ has four vertices -- two of index $2$ and two of index $4$ -- and $6$ edges.  Hence $\pi_1\D$ is free of rank $3$.

Six of the relators of $K_1$ correspond to antipodal pairs of triangles of $\G$ that are collapsed in passing to $\D$, and hence their images in $\D$ are nullhomotopic.   It is an exercise to check that each of the other six relators of $K_1$ in $\G$ map to closed paths in $\D$ which -- possibly after cyclic reduction -- involve each edge of $\D$ precisely once, and hence correspond to primitive words in $\pi_1\D$.  Moreover, the three primitive words corresponding to one antipodal  pair of triangles are all conjugates of the same primitive word $U$, say, while those arising fron the other antipodal pair are all conjugates of $U^{-1}$.  Hence $K_1$ has a homomorphic image $\pi_1\D/\<\<U\>\>\cong F_2$, and so $G_1$ is large, as claimed.

\medskip\noindent $\mathbf{G_2}$.

Note that $W_1$ can be written $AB$, where $A=xyxyxy^3xy^3x$ has order $2$ in $\Z_2*\Z_4$ and $B=yxy^2$.
Hence in $G_1$ we have $A^2=(AB)^2=1$ and so $A,B$ generate a dihedral subgroup of $G_1$.  It follows from this that $(AB^n)^2=1$ in $G_1$ for all $n\in\Z$.  But note that $W_2=AB^3$ in $\Z_2*\Z_4$, so $W_2^2=1$ in $G_1$, whence $G_2$ has $G_1$ as a homomorphic image. The largeness of $G_2$ therefore follows from that of $G_1$.

\medskip\noindent $\mathbf{G_3}$.

Let $U:=(xyxy^{-1})^6$.   Then the normal closure $N$ of $U$ and $y^2$ in $G_3$ has index $24$ and a Schreier transversal consisting of initial segments of the word $U$.  The Reidemeister-Schreier presentation of $N$ obtained using this transversal has generators $U$ and twelve conjugates of $y^2$:
$$V_0:=y^2,~~V_1:=xy^2x^{-1},~~V_2:=xyxy^2x^{-1}y^{-1}x^{-1},$$
$$V_3:=xyxy^{-1}xy^2x^{-1}yx^{-1}y^{-1}x^{-1},~~\mathrm{etc.}$$
There are twelve relators of the form $V_j^2$ arising from the relator $y^4$ of $G_3$, and twelve further relators arising from the relator $W_3^2$.  The alternating exponents of $y$ in the word $U$ mean that the latter relators fall into two distinct patterns according to parity.  These are (with subscripts modulo $12$):
$$R_j:=\left\{ \begin{array}{ll}
V_{j+3}V_{j+4}V_{j+6}V_{j+7}V_{j+8}V_{j+7}V_{j+6}V_{j+3}V_j & (j~\mathrm{even}),\\
 & \\
V_{j+1}V_{j+2}V_{j+5}V_{j+8}V_{j+5}V_{j+4}V_{j+2}V_{j+1}V_j & (j~\mathrm{odd}).
\end{array}\right. $$
If we further factor out the generators $V_j$ for $j\equiv 2,3$ modulo $4$, then we obtain a quotient group which has presentation
$$\<~V_0,V_1,V_4,V_5,V_8,V_9~|~V_j^2=1~(\forall~j), ~V_0V_4V_8=V_9V_5V_1=1~\>,$$
and so is isomorphic to the free product of two copies of the Klein four-group.  It follows that $N$, and hence also $G_3$, is large as claimed.
\end{spf}

\section{The Awkward Sub-case}\label{evenodd}

By Theorems \ref{eveneventhm}, \ref{oddoddthm}, \ref{twotwosthm} and \ref{oddeventhm}, we are reduced to consideration of the case where $x,y$ have even and odd exponent sums in $W$, respectively. We are unable to conclusively verify Conjecture \ref{Rconj} for this subcase, but we can obtain strong lower bounds for the length parameter of any counterexample. 

 In this subcase, the group $G$ has no essential representation of type (a) or (b).  So if $G$ has no non-abelian free subgroup then  $\tau_W(\lambda)=\sqrt2^a(\lambda^2-1)^m$ with $m=k/2$, by Corollary \ref{tauWform}.  Moreover, since $a$ is odd and $(-1)^m\sqrt2^a=\tau_W(0)\le 2$ by Lemma \ref{parity}, it follows that $a=1$.  In other words, precisely one letter of $W$ is equal to $y^2$.

The following observation will be useful later:

\begin{lem}\label{modeight}
If $W$ has length parameter $k=2m$ in $x$ and exponent-sum $\ell$ in $y$, and $\tau_W(\lambda)=\sqrt2(\lambda^2-1)^m$ then $2k+\ell\equiv\pm1$ modulo $8$.
\end{lem}

\begin{proof}
Choose a matrix $Y\in SU(2)$ with trace $\sqrt2$ and let $X:=Y^2$ so that $Tr(X)=0$ and $Tr(XY)=-\sqrt2$.  Then
$$\sqrt2=\tau_W(-\sqrt2)=Tr(W(X,Y))=Tr(Y^{2k+\ell}),$$
which gives the result.
\end{proof}

A useful tool for this case is another consequence of Lemma \ref{parity}.  To explain, it is helpful to rewrite 
\begin{equation}\label{unbalanced}
W=xy^{a(1)}\cdots xy^{a(2m-1)}xy^2,~~(a(j)\in\{1,3\}~~\mathrm{for}~1\le j\le 2m-1)
\end{equation}
(up to conjugacy) in the more balanced form
\begin{equation}\label{balanced}
W=y^{a(-m)} x y^{a(1-m)} x \cdots x y^{a(m-1)} x y^{a(m)},
\end{equation}
where $a(j)\in\{1,3\}$ for all $j$ and in particular $a(-m)=1=a(m)$.

\begin{lem}\label{sums}
If $W$ has the form \eqref{balanced}, then $$\sqrt2\tau_W(\lambda)\equiv\sum_{j\in J} P_j(\lambda)~\mathrm{mod}~4,$$
where $J$ is the set of indices $-m\le j\le m$ for which $a(j)=a(-j)$ and $P_j(\lambda)$ is the trace polynomial of $(xy)^{2j}$. 
\end{lem}

\begin{proof}
Let $$U=x y^{a(1-m)} x \cdots x y^{a(m-1)} x y^{a(m)}$$ and $$V=x y^{a(1-m)} x \cdots x y^{a(m-1)} x.$$  Then since $a(-m)=a(m)$ the standard trace identity \eqref{trid} gives $$\tau_W(\lambda)=\tau_y(\lambda)\tau_U(\lambda) - \tau_V(\lambda).$$
But $\tau_y(\lambda)=\sqrt2$ and $\tau_U(\lambda)\equiv P_m(\lambda)~\mathrm{mod}~2$, by Corollary \ref{OKmod2}.
If $n>0$ is the largest integer less than $m$ for which $a(-n)=a(n)$, then $V$ is conjugate to 
$$\pm y^{a(-n)} x \cdots  x y^{a(n)}$$ and an inductive argument completes the proof.  If no such $n$ exists, then $V$ is conjugate to $\pm y^{a(0)}$ and so $\tau_V(\lambda)=\pm\sqrt2=\pm P_0(\lambda)/\sqrt2$, giving the desired result in this case.
\end{proof}

\begin{lem}\label{Laurent}
If $\tau_W(\lambda)=\sqrt2(\lambda^2-1)^m$, then the set $J$ in Lemma \ref{sums} is precisely the set of integers $j$ for which the coefficient of $t^j$ in the Laurent polynomial $(t+1+t^{-1})^m$ is odd.
\end{lem}

\begin{proof}
It follows from the Cayley-Hamilton Theorem (or by direct calculation) that $A+A^{-1}=\mathrm{Trace}(A)\cdot I$ for any matrix $A\in SL_2(\C)$.  If $X,Y\in SL_2(\C)$ are matrices with $\mathrm{Trace}(X)=0$, $\mathrm{Trace}(Y)=\sqrt2$ and $\mathrm{Trace}(XY)=\lambda$, then $\mathrm{Trace}((XY)^2)=\lambda^2-2$, so
$$\mathrm{Trace}\left(~\left[(XY)^2 + I + (XY)^{-2}\right]^m~\right) =\mathrm{Trace}\left((\lambda^2-1)^m\cdot I\right)=2(\lambda^2-1)^m,$$
and the result follows.
\end{proof}

\begin{lem}\label{div2}
Suppose that $W$ has the form in \eqref{balanced} with $m=2n$ even and $G=\<x,y|x^2=y^4=W^2=1\>$ has no non-abelian free subgroups.  Then there exists another word $U$ say, of the form \eqref{balanced} and $m=n$, such that the  group  $H=\<x,y|x^2=y^4=U^2=1\>$ is a subquotient of $G$ -- and hence also has no non-abelian free subgroups.
\end{lem}

\begin{proof}
Since $x,y$ have even and odd exponent sums in $W$, respectively, there are no essential representations $G\to\P$ of types (a) or (b).  Since $G$ has no non-abelian free subgroups, it follows that the only essential representations $G\to\P$ are of type (c), and hence $\tau_W(\lambda)=\sqrt2(\lambda^2-1)^{2n}$.  By Frobenius' Theorem, $$(t^{-1}+1+t)^{2n}\equiv (t^{-2}+1+t^2)^n~\mathrm{mod}~2,$$ and so by Lemma \ref{Laurent} $a(j)\ne a(-j)$ for all odd $j\in\{-m,\dots,m\}$.

The normal closure $K$ of $y$ in $G$ has a presentation of the form
$$K=\<y_1,y_2|y_1^4=y_2^4=W_1^2=W_2^2=1\>,$$
where $y_1=y$, $y_2=xyx$, $W_1=W$ and $W_2=yxWxy^3$.  In particular
$$W_1=y_1^{a(-2n)}y_2^{a(-2n+1)}\cdots y_2^{a(2n-1)} y_1^{a(2n)}$$
and so $W_1^2$ belongs to the normal closure of $y_1^2$.  Hence $$H:=\<y_1,y_2|y_1^2=y_2^4=W_2^2\>$$ is a subquotient of $G$ satisfying the stated conditions.
\end{proof}

Lemma \ref{div2} together with the analysis of short words in \S \ref{short} reduces the Rosenberger conjecture for generalised triangle groups of type $(2,4,2)$ in the even-odd subcase to those with length parameter  $k=2m$ for odd $m>1$ or $k=2^j$ for $j>2$.   The search for words $W$ with $\tau_W(\lambda)=\sqrt2(\lambda^2-1)^m$ is to some extent facilitated by the formulae for the coefficients of $\lambda^{2m-2}$ and $\lambda^{2m-4}$ in \cite[Lemma 9]{HW}, from which one can deduce the following.

\begin{lem}\label{coeffs}
Let $W$ be a word of the form \eqref{unbalanced}, such that $m\ge 3$ and $\tau_W(\lambda)=\sqrt2(\lambda^2-1)^m$. Then:
\begin{itemize}
\item $a(j)=a(j+1)$ for precisely $m-1$ of the integers $j\in\{1,2,\dots,2m-2\}$; and
\item  $a(j)=a(j+2)$ for precisely $(m-3)/2$ (if $m$ is odd) or $(m-2)/2$ (if $m$ is even) of the integers $j\in\{1,2,\dots,2m-3\}$.
\end{itemize}
\end{lem}

\begin{spf}
In the case of interest, where $(p,q,r)=(2,4,2)$, the expressions in \cite{HW} have a particularly simple form.
Reverting to the form \eqref{unbalanced} for $W$, with $a(2m)=2$, define
$$b(j)=\frac1{\sqrt2}\exp\left(\frac{i\pi(a(j)-a(j+1))}{4}\right)=\frac12\left(1\pm i\right)$$
for $j\in\{2m-1,2m\}$ (indices modulo $2m$), and otherwise
$$b(j)=\exp\left(\frac{i\pi(a(j)-a(j+1))}{4}\right)\in\{1,\pm i\}.$$
Then the coefficient of $\lambda^{2m-2}$ in $\tau_W(\lambda)$ is $-\sqrt2 B$ where
$$B:=\sum_{j=1}^{2m} b(j).$$
Given that, for $1\le j\le 2m-2$, $b(j)=1$ precisely when $a(j)=a(j+1)$, the first part of the result follows by equating the real part of $B$ with $m$.

The coefficient of $\lambda^{2m-4}$ in $\tau_W(\lambda)$ is 
$$\sqrt2\sum_{1<\ell-j<2m-1}b(j)b(\ell)=\frac1{\sqrt2} \left(B^2-\sum_{j=1}^{2m}b(j)^2-2\sum_{j=1}^{2m}b(j)b(j+1)\right)$$
(as before, indices modulo $2m$).  

Now $b(j)^2=\pm i$ when $j\in\{2m-1,2m\}$ and otherwise $b(j)^2=\pm 1$.
Since $B=m$ it follows that $\sum b(j)^2$ has real part $0$, and hence when $\tau_W(\lambda)=\sqrt2(\lambda^2-1)^m$ that $\sum b(j)b(j+1)$ has real part $m/2$.

If $m$ is odd, then $a(2m-1)=a(1)$ by Lemmas \ref{sums} and \ref{Laurent} and so $b(2m-1)b(2m)=1/2$, while 
if $m$ is even then $a(2m-1)\ne a(1)$ and so $b(2m-1)b(2m)=i/2$.  In either case,
each of $b(2m-2)b(2m-1)$, $b(2m)b(1)$ has real part $\frac12$.   For $1\le j\le 2m-3$, $b(j)b(j+1)\in\{1,\pm i\}$, with $b(j)b(j+1)=1$ precisely when $a(j)=a(j+2)$.  The second part of the statement follows.  
\end{spf}

It follows in particular from Lemma \ref{coeffs} that the only candidates for $W$ when $k=6$ or $k=8$ are (up to equivalence)
$$W=yxyxy^3xy^3xyxyxy,$$
$$W=yxyxy^3xy^3xyxyxyxy^3xy,$$
$$W=yxy^3xy^3xyxyxy^3xyxyxy.$$
The resulting groups were shown in \S \ref{short} to have non-abelian free subgroups.  So the desired result follows when the odd part of the length parameter is $1$ or $3$.   To complete the picture, we propose the following conjecture, for which there is some computational evidence, and which implies the Rosenberger Conjecture for groups of type $(2,4,2)$.

\begin{conj}\label{noW}
There is no word $W\in\Z_2*\Z_4$ for which $\tau_W(\lambda)=\sqrt2(\lambda^2-1)^m$ for odd $m\ge 5$.
\end{conj}

The second part of Lemma \ref{coeffs} gives a strong restriction on $W$.  Note also that, if for some $j\ge 0$ the coefficients of $t^j$ and $t^{j+2}$ in $(t+1+t^{-1})^m$ have opposite parity, then either $a(j)=a(j+2)$ or $a(-j-2)=a(-j)$, by Lemma \ref{Laurent}.  Putting these remarks together with an analysis of the coefficients of $(t+1+t^{-1})^m$, one can prove:

\begin{lem}
Conjecture \ref{noW} holds for odd $m$ of one of the forms
\begin{itemize}
\item $2^j-1$ or $2^j-3$ ($j\ge 3$);
\item $3\cdot 2^j-1$ ($j\ge 3$) or $3\cdot 2^j-3$ ($j\ge 2$).
\end{itemize}
\end{lem} 

\begin{proof}[Sketch Proof]
Let $m=2^j-1$.  For $0\le \ell\le m$ define 
$$c(\ell):=c(-\ell)=\left\{\begin{array}{ll} 0~~\mathrm{if}~~m-\ell\equiv 2~~\mathrm{mod}~~3\\ 1~~\mathrm{otherwise}   \end{array}\right.$$
and $P_j(t):=\sum_{\ell=-m}^m c(\ell)t^\ell$.    Then an easy exercise shows that
$$(P_j(t)(t^{-1}+1+t)\equiv t^{-m-1}+1+t^{m+1} \equiv (t^{-1}+1+t)^{m+1}~~\mathrm{mod}~~2.$$
Hence $P_j(t)\equiv (t^{-1}+1+t)^m$ modulo $2$.  But the coefficients of $t^\ell$ and $t^{\ell+2}$ in $P_j(t)$ have opposite parity whenever $0\le \ell\le m-3$ and $m-\ell\equiv 0$ or $1$ modulo $3$.
It is an easy exercise to check that there are 
$$\left\lfloor \frac{2m-5}3\right\rfloor$$
such integers, which is more than $(m-3)/2$ when $j\ge 3$.   By the above remarks and Lemma \ref{coeffs} this gives the desired result for $m$ of the form $2^j-1$.

\medskip
A similar analysis applies when $m=2^j-3=4\times (2^{j-2}-1)+1$.  Then $(t^{-1}+1+t)^m\equiv P_{j-2}(t^4)(t^{-1}+1+t)$ modulo $2$.
The periodicity (with period $3$) in the list of coefficients of non-negative powers of $t$ in $P_{j-1}(t)$ 
(modulo $2$) translates to periodicity of period $12$ in the coefficients of non-negative powers of $t$ in $(t^{-1}+1+t)^m$ (modulo $2$).  Within any given period, the coefficients of  $t^\ell$ and $t^{\ell+2}$  have opposite parity for precisely $6$ values of $\ell$. From this observation and the behaviour of the sequence of coefficients near $\ell=0$ and $\ell=m$, we can calculate that the coefficients of  $t^\ell$ and $t^{\ell+2}$  have opposite parity for precisely $(m-1)/2$ integers $0\le \ell\le m-3$.  Again the result follows from Lemma \ref{coeffs}.

\medskip
The analysis in the remaining cases is again similar.  Since $3\cdot 2^j=2^{j+1}+2^j$ we deduce that
$$(t^{-1}+1+t)^m=(t^{-2^{j+1}}+1+t^{2^{j+1}})(t^{-1}+1+t)^n$$
where $n=2^j-1$ or $n=2^j-3$ respectively.  Thus the list of coefficients of $(t^{-1}+1+t)^m$ modulo $2$ consists of three copies of the corresponding list for $(t^{-1}+1+t)^n$ modulo $2$ -- consecutive copies being separated by one or five zeroes respectively.  This allows us to calculate the number of integers $0\le \ell\le m-3$ for which the coefficients of $t^\ell$ and $t^{\ell+2}$ have opposite parity, which turns out to be $(m-1)/2$ when $m=3\cdot 2^j-3$, and $(2m-7)/3$ (resp. $(2m-10)/3$) when $m=3\cdot 2^j-1$ with $j$ odd (resp. even). Once again Lemma \ref{coeffs} applies to give the result.
\end{proof}

\begin{cor}\label{thirtythree}
Conjecture \ref{noW} holds for odd $m\le 15$.  Hence the Rosenberger Conjecture holds for generalised triangle groups of type $(2,4,2)$ with length parameter $k\le 33$.
\end{cor}

\begin{proof}[Sketch Proof]
Every odd number between $5$ and $15$ satisfies one of the conditions of the lemma, except for $11$.

Suppose that $m=11$, and that $\tau_W(\lambda)=\sqrt2(\lambda^2-1)^{11}$. Write $W$ in the form \eqref{balanced}.  Then there are precisely $(m-3)/2=4$ integers $\ell$ between $1-m$ and $m-3$ for which $a(\ell)=a(\ell+2)$, by Lemma \ref{coeffs}.  By Lemma \ref{Laurent} these must be four of the eight values of $\ell$ for which the coefficients of $t^\ell$ and $t^{\ell+2}$ in $(t^{-1}+1+t)^{11}$ have opposite parity, namely:
\begin{itemize}
\item one of $-7,5$;
\item one of $-6,4$;
\item one of $-4,2$; and
\item one of $-3,1$.
\end{itemize} 
In particular $a(-5)\ne a(-3)=a(3)\ne a(5)$. 

In a similar way, one can consider which $m-1=10$ of the integers $\ell$ between $1-m$ and $m-2$ satisfy $a(\ell)=a(\ell+1)$, in the light of the coefficients of $(t^{-1}+1+t)^{11}$.  We can deduce that precisely one of the following pairs of equalities hold:
\begin{itemize}
\item $a(-6)=a(-5)$ and $a(5)=a(6)$;
\item $a(-3)=a(-2)$ and $a(2)=a(3)$.
\end{itemize}
Since $a(3)\ne a(5)$, it follows that $a(2)=a(6)$.  Hence we have
$$a(-10)\ne a(-8)\ne a(-6)=a(-2)\ne a(0)\ne a(2)=a(6)\ne a(8)\ne a(10).$$

 In combination with the facts that  $a(\pm5)\ne a(\pm3)$, $a(\pm11)=1$, and $a(-j)\ne a(j)$ for $j=1,4,7,9$, this tells us that the exponent-sum of $y$ in $W$ is either $41$ (if $a(0)=3$) or $47$ (if $a(0)=1$).
But this, together with the fact that $2k=44\equiv 4$ modulo $8$, contradicts Lemma \ref{modeight}.
\end{proof}

\section{Computer searches}\label{it}

The results of \S\S\,\ref{charvar} - \ref{short} confirm the Rosenberger Conjecture for generalised triangle groups of type $(2,4,2)$ in the three sub-cases where the exponent-sums of $x,y$ in $W$ are odd-odd, odd-even or even-even.  Those of \S \ref{evenodd} confirm it in the remaining sub-case under a bound on the length parameter -- see Corollary \ref{thirtythree}.  While computational methods were used in experimental investigations leading to these results, their final proofs can be written down and confirmed by purely theoretical means.

If one is willing to accept more widespread use of computational methods as part of a proof, then one can extend the bounds on Corollary \ref{thirtythree} by using computer searches to eliminate candidate words as possible counterexamples.  
 The amount by which the bounds can be improved is restricted by the fact that any search algorithm is at least exponential-time in the length of words being analysed.   On the other hand, the searches can be streamlined by judicious use of the results in \S \ref{evenodd}.   For example, a na\"{\i}ve search for words of length $k=2m$ with trace polynomial $\sqrt2(\lambda^2-1)^m$ would list all $2^{2m-1}$ words of the form \eqref{balanced} (which one can reduce by a factor of $4$ to $2^{m-3}$ using the group-isomorphisms induced by the moves $W\leftrightarrow W^{-1}$ and $y\leftrightarrow y^{-1}$).  One would then calculate for each word the trace polynomial and compare it to the target polynomial $\sqrt2(\lambda^2-1)^m$.

However, Lemma \ref{Laurent} tells us how to derive almost half of the letters of our candidate word from the other half.  So this reduces the number of words to be considered from $2^{2m-1}$ to $2^m$ (or $2^{m-2}$ using the group-isomorphism moves as before).  This extends by approximately a factor of $2$ the word-length that can be searched in a given time period.

A further streamlining of the search algorithm is provided by Lemmas \ref{modeight} and \ref{coeffs}: counting exponent-sums, and the numbers of $j$ for which $a(j)=a(j+2)$ or $a(j)=a(j+1)$, is much more computationally efficient than computing the trace polynomial.

Using the above tricks, we were able to confirm the following by computer searches on a basic laptop using a standard issue of GAP \cite{GAP} within a period of a few hours for any given length parameter.

\begin{clem}
Conjecture \ref{noW} holds for odd $m\le 31$.
\end{clem}

\begin{ccor}
The Rosenberger Conjecture holds for generalised triangle groups of type $(2,4,2)$ with length parameter $k\le 65$.
\end{ccor}

One method of showing that a group has free subgroups is to use small-cancellation theory.   In particular, it is not difficult to show that, if the square of a word $W$ of the form \eqref{balanced} satisfies the small cancellation condition C6, then it has non-abelian  free subgroups -- for example that generated by $(xy)^N$ and $(xy^3)^N$ for sufficiently large $N$.  However, neither $yxy^2$ nor $y^3xy^2$ is a piece in $W^2$, so the failure of C6 entails the existence of a piece of length at least $2m-1$.  Indeed, a more precise analysis shows that -- possibly after inverting $W$ -- the subword
$U:=xy^{a(2-m)}x\cdots xy^{a(-1)}x$ of syllable length $2m-3$ is the initial part of a subword $V$ of syllable length $2m-3+j$, say, which is either
\begin{itemize}
\item equal to its own inverse in $\Z_2*\Z_4$; or
\item periodic of period $j>0$.
\end{itemize}
In the second case, if $U$ is the initial part of a subword longer than $V$ that is periodic of period $j$, then if follows from Lemma \ref{Laurent} that the subsequence $c(-j),c(1-j),\dots,c(j-1),c(j)$ of subscripts of the coefficients $c(j)$ of $t^j$ in $(t+1+t^{-1})^m$ modulo $2$ is periodic of period $j$.  An analysis of the sequence $\{c(j)\}$ leads to a contradiction for $j>1$.  And periodicity of period $1$ is ruled out by Lemma \ref{coeffs}.  So $U$ is of maximal length. The complementary subword $U':=xy^{a(1)}x\cdots xy^{a(m-2)}x$ cannot also be the last part of a periodic subword, for then the whole word $\widehat{U}:=xy^{a(1-m)}x\cdots xy^{a(m-1)}x$ is periodic - again leading to a contradiction.  Thus inverting $W$ again reverts us to the first case.

If $U$ is the start of a subword of syllable length  $2m-3+2j$ which is equal to its own inverse, then combining this fact with Lemma \ref{Laurent} we can show that $j$ is even and $W$ is determined (up to the usual equivalence) by $(m-1-j)/2$ of its letters.   So now our search space is further reduced from $2^{m-2}$ candidate words to $$\sum_{i=0}^{(m-3)/2}2^i=2^{(m-1)/2}-1.$$

Using these remarks, we were able to optimise our computer search and further increase the length bound for the Rosenberger Conjecture.

\begin{clem}
The Rosenberger Conjecture holds for generalised triangle groups of type $(2,4,2)$ with length parameter $k\le 97$.
\end{clem} 

(Note however that this last result says nothing about Conjecture \ref{noW} -- a counterexample to that conjecture which satisfied the C6 property would not be found by this search.)


\begin{thebibliography}{99}
\bibitem{BBK} O. A. Barkovich and  V. V. Beniash-Kryvets, {\em On the Tits Alternative for some generalised triangle groups of type $(3,4,2)$} (Russian), Dokl. Nats. Akad. Nauk Belarusi {\bf 47} (2003), no. 6, 24--27.  
\bibitem{BBK2} O. A. Barkovich and V. V. Benyash-Krivets,
{\em On {T}its
alternative for generalised triangular groups of (2,6,2) type} ({R}ussian),
{Dokl. Nats. Akad. Nauk. Belarusi} {\bf 48} (2003), No. 3, 28--33.     
\bibitem{BK03} V. V. Beniash-Kryvets, {\em On the Tits alternative for some finitely generated groups} (Russian),
 Dokl. Nats. Akad. Nauk Belarusi {\bf 47}, No. 2 (2003), 29--32.     
\bibitem{BK} V. V. Benyash-Krivets, {\em On free subgroups of certain generalised
triangle groups} ({R}ussian), Dokl. Nats. Akad. Nauk. Belarusi {\bf 47}, No. 3
(2003) 14--17.
\bibitem{BK2} V. V. Benyash-Krivets,
{\em On {R}osenberger's conjecture for generalized triangle groups of types $(2, 10, 2)$ and $(2, 20, 2)$},
in: {Proceedings of the international conference on mathematics and its applications}
{(S. L. Kalla and M. M. Chawla, eds.)},
Kuwait Foundation for the Advancement of Sciences (2005),
pp. 59--74.                                                                                              
\bibitem{BK07} V. V. Beniash-Kryvets, {\em Free subgroups of generalised triangle groups of type $(2,4,2)$} (Russian),
 Dokl. Nats. Akad. Nauk Belarusi {\bf 51}, No. 4 (2007), 29--32.                                                                                          
\bibitem{BKB} V. V. Beniash-Kryvets and O. Barkovich, {\em On the Tits alternative for some generalized
triangle groups}, Algebra and Discrete Mathematics {\bf 2} (2004), 15--35.              
\bibitem{Bu} J. Button, {\em Large groups of deficiency $1$}, Israel J. Math. {\bf 167} (2008) 111--140.
\bibitem{FLR} B. Fine, F. Levin and G. Rosenberger,
{\em Free subgroups and decompositions of one-relator products of
cyclics. {I}. {T}he {T}its alternative},
{Arch. Math. (Basel)}
{\bf 50}
{(1988)},
{97--109}.
\bibitem{FRR} B. Fine, F. Roehl and G. Rosenberger,
{\em The {T}its alternative for generalized triangle groups.},
in: {Groups - Korea '98. Proceedings of the 4th international conference, Pusan, Korea, August 10-16, 1998}
(Y.-G. Baik et al, eds.), Walter de Gruyter, Berlin (2000), pp. 95--131.
\bibitem{GAP} The GAP Group, GAP -- Groups, Algorithms, and Programming, Version 4.11.1; 2021. (https://www.gap-system.org) 
\bibitem{Gold} W. M. Goldman, {\em An exposition of results of Fricke},  	arXiv:math/0402103 (2004).
\bibitem{Hi} G. Higman, {\em A finitely generated infinite simple group}, J. London Math. Soc. {\bf 26} (1951) 61-64.
\bibitem{Ho10} J. Howie, {\em Generalised triangle groups of type (3,3,2}), arXiv:math/1012.2763.
\bibitem{Ho13} J. Howie,  {\em Generalised triangle groups of type (3,q,2)}, 
Algebra and Discrete Mathematics {\bf 15} (2013), 1--18.
\bibitem{HK} J. Howie and O. Konovalov, {\em Generalised triangle groups of type (2,3,2) with no cyclic essential representations}, to appear in: Festschrift Gerhard Rosenberger, (eds. V. Diekert and M. Kreuzer), De Gruyter (2024).
See also arXiv:math/1612.00242.
\bibitem{HK-code} J. Howie and O. Konovalov, {\em Generalised triangle groups of type (2,3,2) with no cyclic essential representations.
Supplementary code}, v1.0.0, Zenodo, 2023, \url{https://doi.org/10.5281/zenodo.10401274},
GitHub: \url{https://github.com/olexandr-konovalov/generalised-triangle-groups-232}.
\bibitem{HW} J. Howie and G. Williams, {\em Free subgroups in certain generalized triangle groups of type $(2,m,2)$},  Geom. Dedicata {\bf 119} (2006), 181--197.
\bibitem{HW342} J. Howie and G. Williams, {\em The Tits alternative for generalized triangle groups of type (3,4,2)}, Algebra and Discrete Mathematics  (2008) No. 4, 40-48,  arXiv:math/0603682 (2006).
\bibitem{LR} F.Levin and G.Rosenberger, {\em On free subgroups of generalized triangle groups, part II}, In Group theory (Granville, OH, 1992), pages 206--228. World Sci. Publishing, River Edge, NJ, 1993.
\bibitem{Ros} G. Rosenberger, {\em On free subgroups of generalized triangle groups}, Algebra Logika {\bf 28} (1989) 227--
240.
\bibitem{Wil} A. G. T. Williams, {\em Studies on generalised triangle groups}, Ph. D thesis, Heriot-Watt University, 2000.
\bibitem{Will2} A. G. T. Williams,
{\em Generalized triangle groups of type $(2,m,2)$} in:
{Computational and Geometric Aspects of Modern Algebra, {LMS} {L}ecture {N}ote {S}eries 275},
{(M. Atkinson et al, eds.)},
{Cambridge University Press}
(2000),
pp. 265-279.
\end{thebibliography}
\end{document}